\documentclass[twoside, 10pt]{article}
\usepackage{fullpage}
\usepackage{pgf}
\usepackage{graphicx}
\usepackage{amsmath}
\usepackage{amsfonts}
\usepackage{amssymb}
\usepackage{amsxtra}
\usepackage{amstext}
\usepackage{latexsym}
\usepackage{dsfont}
\usepackage{stmaryrd}
\usepackage{mathrsfs}
\usepackage{bbm,bm}
\usepackage{soul}

\usepackage{color,graphicx,times}

\newtheorem{theorem}{Theorem}

\newtheorem{corollary}[theorem]{Corollary}

\newtheorem{lemma}[theorem]{Lemma}

\newtheorem{proposition}[theorem]{Proposition}
\newtheorem{remarks}[theorem]{Remarks}

\newenvironment{example}[1][Example] {\noindent\textbf{#1 } }{}
\newenvironment{proof}[1][\noindent Proof]{\textbf{#1} }{\hfill \ \rule{0.5em}{0.5em}} \newenvironment{remark}[1][\noindent
Remark]{\noindent\textbf{#1} }

\def\bin#1,#2{#1\choose#2}
\def\gra{\mathrm{gra}}

\def\qq{\mathbb{Q}}
\def\rr{\mathbb{R}}

\usepackage{wrapfig}
\usepackage[all]{xy}
\usepackage{tikz}
\usepackage{hyperref}
\usetikzlibrary{patterns}
\newcommand{\tos}{\rightrightarrows} % point-to-set mappings
\title{\bf Inverse maximum theorems and some consequences}

\author{John Cotrina$^{1,2}$\and Ra\'ul Fierro$^{1,2}$\thanks{Corresponding author: Instituto de Matem\'atica, Pontificia Universidad Cat\'olica de
Valpara\'{\i}so, Casilla 4059, Valpara\'{\i}so, Chile. Email: raul.fierro@pucv.cl; rafipra@gmail.com}\and \\[1ex]
{\footnotesize $^{1}$Instituto de Matem\'atica. Pontificia Universidad Cat\'{o}lica de Valpara\'{\i}so.}\\[-0.5ex] {\footnotesize $^{2}$Instituto de Matem\'aticas. Universidad de Valpara\'{\i}so.}
\and John Cotrina \thanks{Universidad del Pac\'ifico, Lima, Per\'u. Email: cotrina\_je@up.edu.pe}}

\author{John Cotrina$^{1}$\thanks{Email: cotrina\_je@up.edu.pe}\\
{\footnotesize $^{1}$Universidad del Pac\'ifico, Lima, Per\'u.}\\[1.5ex]
Ra\'ul Fierro$^{2,3}$\thanks{Email: raul.fierro@pucv.cl; raul.fierro@uv.cl}\\
{\footnotesize $^{2}$Instituto de Matem\'atica. Pontificia Universidad Cat\'{o}lica de Valpara\'{\i}so.}\\[-0.5ex] {\footnotesize $^{3}$Instituto de Matem\'aticas. Universidad de Valpara\'{\i}so.}
}

\begin{document}
\maketitle
\begin{abstract}
We deal with inverse maximum theorems, which  are inspired by the ones given by Aoyama, Komiya, Li \emph{et al.}, Park and Komiya,  and Yamauchi. As a consequence of our results, we state and prove an inverse maximum Nash theorem and show that any generalized Nash game can be reduced to a classical Nash game, under suitable assumptions. Additionally, we show that  a result by Arrow and Debreu, on the existence of solutions for generalized Nash games, is actually equivalent to the one given by Debreu-Fan-Glicksberg for classical Nash games, which in turn is equivalent to Kakutani-Fan-Glisckberg's fixed point theorem. %Other equivalences are stated in the setting of normed spaces.
\end{abstract}

\noindent{\textbf{Keywords:} Inverse maximum theorems; Berge's maximum theorem;  Generalized Nash game; Kakutani's fixed point theorem.}

\noindent{{\bf MSC (2020)}: 46N10; 91B50;  49J35

\section{Introduction}

In 1959, Berge provided conditions for the continuity of its maximizers' correspondence  with respect to its parameters. In 1997, Komiya \cite{Ko97} proposed the inverse problem, which consists of finding conditions on the correspondence for obtaining a continuous  function defining it as its maximizers' correspondence.
In that sense, Komiya gave a positive answer to the previous problem on finite-dimensional spaces and he showed the equivalence between the Kakutani theorem and an existence theorem for maximal elements by Yannelis and Prabhakar in \cite{YP83}. In 2001, Park and Komiya \cite{PK01} extended a certain class of correspondences defined from a topological space to a metric topological vector space  with convex balls. In 2003, Aoyama \cite{Ao03} worked on convex metric spaces and used the same class of correspondences. In 2008, Yamauchi  \cite{Ya08} also gave an affirmative answer to this inverse problem on locally convex topological vector spaces, but he deals with upper semicontinuous correspondences whose graphs are $G_\delta$ sets. Recently in 2022, Li \emph{et al. }  \cite{LiLiFeng2022}  presented a similar result to the one given by Yamauchi, but they considered the correspondence with a compact and convex range instead the assumption of the $G_\delta$ graph. Moreover, they proved that Kakutani-Fan-Glicksberg's fixed point theorem is a consequence of the generalized coincidence point theorem, which is a consequence of the equilibrium theorem for generalized Nash games. Inspired by these works we also present some inverse maximum theorems.

By making use of the finite-dimensional inverse maximum theorem by Komiya in \cite{Ko97}, in 2016, Yu \emph{et al.}  \cite{YuEtAl16} proved that the Kakutani and Brouwer fixed point theorems, among other important results, can be obtained from the Nash equilibrium theorem. Recently in 2021, Bueno and Cotrina \cite{BC-2021} used an inverse maximum result  to reformulate quasi-equilibrium problems and quasi-variational inequalities as generalized Nash games on Banach spaces. It is important to mention that the converse reformulation is known, see for instance \cite{AuEtAl21,BCC21,JCAHAS,CS21,JC-JZ-PS,Facchinei2007,HARKER199181} and the references therein. Motivated by these works, among the main applications of  our inverse maximum theorems, we state and prove an inverse maximum Nash theorem, which consists of finding appropriate payoff functions that give solutions  a predetermined strategy set. We think this formulation could be appropriate for economical agents having non-modifiable strategies due to some restrictions. We also reformulate generalized Nash games as classical Nash games, under suitable assumptions. Moreover, this reformulation allows us to show the equivalence between the Arrow and Debreu result in \cite{AD54} and the Debreu-Fan-Glicksberg theorem on locally convex spaces. We remark that some authors have generalized some equilibrium theorems by relaxing the continuity assumption of the payoff functions, see for instance \cite{AuEtAl21,Dasgupta,MORGAN2007,Reny,Tian95}. However, their existence results are in the finite-dimensional setting and are a consequence of the Kakutani-Fan-Glicksberg theorem, and, according to the equivalence proved in this work, they are equivalent.
 By taking advantage of the fact that the inverse maximum theorems, presented in the current work, are valid in the infinite-dimensional setting, we prove  Kakutani-Fan-Glicksberg's fixed point theorem \cite{Fa52,Gl52} and the Debreu-Glicksberg-Fan theorem on locally convex topological vector spaces, which generalizes the equivalences proved by Yu \emph{et al.}  \cite{YuEtAl16}.

We subdivided this work as follows. We introduce in Section \ref{Preliminaries} some definitions and facts.
In Section \ref{Results}, we  present our main results. Section \ref{Applications} is   devoted to generalized Nash games and Section \ref{FPT} is devoted to fixed point theory. Finally, we summarize the major results of this work in Section \ref{Conclusions}.

\section{Preliminaries}\label{Preliminaries}

A real-valued function $f:C\to\rr$ on a convex set $C$ in a vector space is said to be \emph{quasi-concave} if for each $\lambda\in\rr$ the set $\{x\in C:~f(x)\geq\lambda\}$ is convex. Clearly, the function $f$ is quasi-concave if, and only if, $f(tx+(1-t)y)\geq \min\{f(x),f(y)\}$ for all $x,y\in C$ and all $t\in[0,1]$.

Let $X$ and $Y$ be two non-empty sets and $\mathcal{P}(Y)$ be the family of all subsets of $Y$. A \emph{correspondence} or \emph{set-valued map} $T:X\tos Y$ is an application $T:X\to \mathcal{P}(Y)$, that is, for $u\in X$, $T(u)\subseteq Y$.
The \emph{graph} of $T$ is defined as
\[\gra(T)=\big\{(u,v)\in X\times Y\::\: v\in T(u)\big\}.\]

We now recall the notion of continuity for correspondences. Let $X$ and $Y$ be two topological spaces. A correspondence $T:X\tos Y$ is said to be:
\begin{itemize}
 \item \emph{closed}, when $\gra(T)$ is a closed subset of $X\times Y$;
 \item \emph{lower semicontinuous} if the set $\{x\in X\::\: T(x)\cap G\neq \emptyset\}$ is open, whenever $G$ is open;
 \item \emph{upper semicontinuous} if
the set $\{x\in X\::\: T(x)\cap F\neq \emptyset\}$ is closed, whenever $F$ is closed; and
 \item \emph{continuous} if it is both  lower and upper semicontinuous.
 \end{itemize}

It is straightforward to verify that a correspondence $T$ is lower semicontinuous if, and only if, for all $x\in X$ and any open set $G\subseteq Y$, with $T(x)\cap G\neq\emptyset$, there exists a neighborhood $V_x$ of $x$ such that $V_x$ such that $T(x')\cap G\neq\emptyset$ for all $x'\in V_x$. In a similar way, $T$ is upper semicontinuous if, and only if,  for all $x\in X$ and any open set $G$, with $T(x)\subseteq G$,  there exists a neighborhood $V_x$ of $x$  such that $T(V_x) \subseteq V$.

%Given a real number, $a$, we denote by $[a]$ the integer part of $a$.

From now on, we will assume that any topological space is Hausdorff.
\section{Main results}\label{Results}

Let us consider a correspondence $K:X\tos Y$  and a function $\theta:\gra(K)\to\rr$, where $X$ and $Y$ are two topological spaces. We associate with them  the  argmax correspondence $M_0:X\tos Y$ defined as
\[
M_0(x)=\left\lbrace y\in K(x): \theta(x,y)=\sup_{z\in K(x)}\theta(x,z)\right\rbrace.
\]
The Berge maximum theorem can be stated as follows.
\begin{theorem}\label{t1}
If $K$ is  continuous with non-empty  compact values and $\theta$ is continuous. Then,   the argmax correspondence $M_0$, is upper semicontinuous and has non-empty compact values. Moreover, the function $m:X\to\rr$ defined as $m(x)=\sup_{y\in K(x)}\theta(x,y)$ is continuous.
\end{theorem}
Now, if $K$ is  continuous  and $M:X\tos Y$ is a correspondence such that $\gra(M)\subseteq\gra(K)$, then does there exists a  function $\theta:\gra(K)\to\rr$ such that $M$ is the argmax correspondence associated to $K$ and $\theta$? A first answer to this question was given by  Komiya in \cite{Ko97}, in the linear and finite dimensional setting. Inspired from this, but without considering convexity properties of $\theta$, we give a positive answer on topological spaces.

\begin{theorem}\label{t5}  Let $X$ and $Y$ be two  topological spaces, $K:X\tos Y$ be a continuous correspondence with normal graph, compact and non-empty values, and   $M:X\tos Y$ be  a correspondence with  non-empty values such that $M(x)\subseteq K(x)$, for all $x\in X$.
Then, the following two conditions are equivalent:
\begin{itemize}
\item [(a)] there exists a continuous function $\theta:\gra(K)\to[0,1]$, such that
$$
\gra(M)=\left\{(x,y)\in\gra(K):\theta(x,y)=\sup_{z\in K(x)}\theta(x,z)\right\},
$$
and
\item [(b)] $\gra(M)$ is a closed and $G_\delta$ set.
\end{itemize}
\end{theorem}
\begin{proof}
Let $m:X\to\rr$ be the function defined by $m(x)=\sup_{z\in K(x)}\theta(x,z)$ and suppose condition (a) holds.
%The Berge maximum theorem
Theorem \ref{t1} implies that $m$ is continuous, and consequently, $M$ is closed. Moreover,
\[
\gra(M)=\bigcap_{n=1}^{\infty}\{(x,y)\in \gra(K):~ \theta(x,y)>(1-1/n)m(x)\},
\]
which proves that $\gra(M)$ is a $G_\delta$ set and hence
 condition (b) holds.

Next, suppose there exists a non-increasing sequence of open subsets of $X\times Y$, $\{U_n\}_{n\in\mathbb{N}}$, such that $\gra(M)=\bigcap_{n\in\mathbb{N}}U_n\cap \gra(K)$. Since $\gra(K)$ is normal, for each $n\in\mathbb{N}$, there exists a Urysohn function $\theta_{n}:\gra(K)\to[0,1]$ such that $\theta_{n}\equiv 1$ on $\gra(M)$ and $\theta_{n}\equiv 0$ on $\gra(K)\setminus U_n$. Let $\theta:\gra(K)\to[0,1]$ be defined as
$$
\theta(x,y)=\sum_{n=1}^\infty\frac{1}{2^n}\theta_{n}(x,y).
$$
We have, $\theta$ is continuous. Moreover, $\theta^{-1}(\{1\})=\gra(M)$ and, since $M$ is non-empty valued, it follows that
$$
\gra(M)=\left\{(x,y)\in\gra(K):\theta(x,y)=\sup_{z\in K(x)}\theta(x,z)\right\}.
$$
Thus, the proof is complete.
\end{proof}

In order to guarantee the normality  of $\gra(K)$, we can assume for instance that $X\times Y$ is a normal space and $\gra(K)$ is closed. However, this condition is  not necessary, as we can see if $X$ is a normal space, $Y$ is a non-empty non-normal space and $K:X\tos Y$ is defined by $K(x)=\{y_0\}$, for all $x\in X$. Clearly, $X\times Y$ is not normal  but $\gra(K)$ is.

\begin{remark}\rm
According to Theorem \ref{t1}, the correspondence $M$, in Theorem \ref{t5}, is upper semicontinuous  with non-empty and compact values, whenever  the two equivalent conditions hold.

On the other hand, the function $\theta$ given en Theorem \ref{t5} is not unique. It is enough to see that for any strictly increasing function $h:\rr\to\rr$, the function $\vartheta=h\circ\theta$ satisfies Theorem \ref{t5}.
\end{remark}

The following result is an inverse maximum theorem and it is also a generalization of Lemma 4.1 in \cite{BC-2021}.
\begin{proposition}\label{t7}
Let $X$ and $Y$ be two topological spaces such that $X\times Y$ is normal and  $M:X\tos Y$ be  a correspondence with  non-empty and compact values. If $M$ is upper semicontinuous and its graph is a $G_\delta$ set, then
there exists a continuous function $\theta:X\times Y\to[0,1]$ such that
\[
\gra(M)=\left\{(x,y)\in X\times Y:\theta(x,y)=\sup_{z\in Y}\theta(x,z)\right\}.
\]
\end{proposition}
\begin{proof}
From the fact that any upper semicontinuous correspondence with compact values is closed, this proof follows by the same steps of the proof of (b) implies (a), in Theorem \ref{t5}.
\end{proof}

The following result gives sufficient conditions in order to guarantee the inverse of Berge's maximum theorem in the context of topological spaces without normality assumption.

\begin{theorem}\label{P-3}
Let $X$ and $Y$ be two topological spaces, and $K,M:X\tos Y$ be two correspondences such that $M(x)$ is non-empty  and $M(x)\subseteq K(x)$, for all $x\in X$. Suppose there exists a family of open sets in $\gra(K)$, $\{U_t\}_{t>0}$, such that
\begin{itemize}
\item[(i)] $\bigcup_{t>0}U_t=\gra(K)$,
\item[(ii)] $\overline{U}_s\subseteq U_t$, for all $s<t$, and
\item[(iii)] $\gra(M)=\bigcap_{t>0} U_t$.
\end{itemize}
Then, there exists a continuous function $\theta:\gra(K)\to[0,1]$ such that
\[
\gra(M)=\left\{(x,y)\in \gra(K):\theta(x,y)=\sup_{z\in K(x)}\theta(x,z)\right\}.
\]
\end{theorem}
\begin{proof}
Thanks to conditions (i) and (ii), by Lemma 3, Chapter 4 in \cite{Ke55}, the function $\tau:X\times Y\to\rr$ defined as
\[
\tau(x,y)=\inf\{t>0:~(x,y)\in  U_t\}
\]
is continuous. Consequently, the function $\theta$ defined as
\[
\theta(x,y)=1-\tau(x,y)\wedge1
\]
is also continuous. Moreover, we can see that $\theta(x,y)=1$ if, and only if, $\tau(x,y)=0$, which in turn is equivalent to $(x,y)\in U_t$, for all $t>0$. This allows us to conclude that $\theta(x,y)=1$ if, and only if, $(x,y)\in M$.
Finally, the result follows from the fact that $M$ is non-empty valued.
\end{proof}

\begin{remarks}\rm
A few remarks are needed.
\begin{enumerate}
\item Notice that in above result we do not require $M$ to be compact-valued.
\item
Conditions (ii) and (iii), in the previous result, imply that $\gra(M)$ is a $G_\delta$ set.
 Indeed, for each $t>0$, let $q(t)\in\qq$ such that $0<q(t)<t$. Hence, $\gra(M)\subseteq U_{q(t)}\subseteq U_t$ and accordingly
\[
\gra(M)\subseteq\bigcap_{q\in\qq\cap(0,\infty)} U_{q}\subseteq\bigcap_{t>0} U_{q(t)}\subseteq \bigcap_{t>0} U_t.%\subseteq \hat{M}.
\]
Therefore, $\gra(M)=\bigcap_{q\in\qq\cap(0,\infty)} U_{q}$ is a $G_\delta$ set.

\item Theorem \ref{P-3} fails to be true if  the correspondence $M$  has empty values. Indeed,
consider $M:\rr\tos\rr$ such that its graph is $\{(0,0)\}$ and the family of sets $\{B(0,t)\}_{t>0}$, where $B(0,t)$ is the open ball center at $0$ and radius $t$. It is clear that this family of sets satisfies the assumptions of Theorem \ref{P-3}. The function $\theta$ given in the proof is continuous, but the conclusion does not hold, because  $M(1)=\emptyset$ and
$\{y\in\rr:~ \theta(1,y)=\max_{z\in \rr}\theta(1,z)\}=\rr$.
\end{enumerate}
\end{remarks}

As an important consequence of the previous result, we have the following corollary.
\begin{corollary}\label{closed-function}
Let $(X,d_X)$ and $(Y,d_Y)$ be two metric spaces, and $K,M:X\tos Y$ be two correspondences such that $M(x)$ is non-empty and $M(x)\subseteq K(x)$, for all $x\in X$. Suppose $\gra(M)$ is closed in $\gra(K)$ with the metric $d:\gra(K)\to\rr$ defined as $d((x,y),(u,v))=d_X(x,u)+d_Y(y,v)$.  Then, there exists a continuous function $\theta:\gra(K)\to[0,1]$ such that, for all $x\in X$,
\[
M(x)=\{y\in Y:~\theta(x,y)=\sup_{z\in Y}\theta(x,z)\}.
\]
Moreover, for each $x\in X$ and $t>0$, we have
\[
U_t(x)=\bigcup_{x'\in B(x,t)}R_{t-d_X(x,x')}(x'),
\]
where $U_t:X\tos Y$ is the correspondence with $\gra(U_t)=\{(x,y)\in X\times Y:~d((x,y),\gra(M))<t\}$, $B(a,r)=\{z\in X:~d_X(a,z)<r\}$,  and $R_s(z)=\{y\in Y:~d_Y(y,M(z))<s\}$, for all $a,z\in X$ and $r,s>0$.
\end{corollary}
\begin{proof}
It is clear that the family of open sets, $\{U_t\}_{t>0}$, satisfies  all assumptions of Theorem \ref{P-3}. Hence, the existence of function $\theta$ satisfying the above conditions follows. The last part holds by noticing that  $y$ is an element of $U_t(x)$ if, and only if, there exists $(x_0,y_0)\in\gra(M)$ such that $d((x,y),(x_0,y_0))<t$, which is equivalent that $d_X(x,x_0)<t$ and $d_Y(y,y_0)<t-d_X(x,x_0)$.
\end{proof}

The conclusion of Theorem \ref{P-3} can be improved when the range space of the correspondence is a vector space.

\begin{theorem}\label{P-4}
Let $X$ and $Y$ be two  topological spaces, with $Y$ a vector space;    $K,M:X\tos Y$ be two correspondences such that $M(x)$ is non-empty  and $M(x)\subseteq K(x)$, for all $x\in X$.
Suppose there exists an increasing family, $\{U_t\}_{t\geq0}$, of open sets in $\gra(K)$ such that

\begin{itemize}
\item[(i)] $\bigcup_{t>0}U_t=\gra(K)$,
\item[(ii)] $\overline{U}_s\subseteq U_t$, for all $s<t$,
\item[(iii)] $\gra(M)=\bigcap_{t>0} U_t$, and
\item [(iv)] $\{y\in Y: (x,y)\in U_t\}$ is convex, for all $t>0$ and $x\in X$.
\end{itemize}
Then, there exists a continuous function $\theta:\gra(K)\to[0,1]$ such that the following two conditions hold:
\begin{itemize}
  \item [(v)] $\gra(M)=\{(x,y)\in \gra(K):~\theta(x,y)=\sup_{z\in K(x)}\theta(x,z)\}$, and
  \item [(vi)] $\theta(x,\cdot)$ is quasi-concave, for all $x\in X$.
\end{itemize}
\end{theorem}
\begin{proof}
Thanks to Theorem \ref{P-3}, condition (v) holds  and, by applying Lemma 2, Chapter 4 in \cite{Ke55}, for all $s\in (0,1]$, we deduce
\[
\{(x,y)\in \gra(K):~\theta(x,y)\geq s\}= \bigcap_{t>1-s}U_t.
\]
Hence, for each $x\in X$  we have
\[
\{y\in Y:~ \theta(x,y)\geq s\}=\left\lbrace \begin{array}{cc}
\bigcap_{t>1-s}\{y\in Y: (x,y)\in U_{t}\},&0<s\leq 1,\\
Y,& s\leq 0,\\
\emptyset,&s>1.
\end{array}\right.
\]
Therefore, condition (vi) holds and the proof is complete.
\end{proof}

The following example shows that Theorem \ref{closed-function} is neither a consequence of Theorem 3.5 in \cite{LiLiFeng2022} by Li \emph{et al.} nor Theorem 1.3 in \cite{Ya08}  by Yamauchi.

\begin{example}
We consider the correspondence $M:\rr\to\rr$ defined by
\[
M(x)=\left\lbrace\begin{matrix}
[-1/|x|,1/|x|],&x\neq0\\
\rr,&x=0.
\end{matrix}\right.
\]
For each $t>0$, we define the correspondence $U_t:\rr\tos\rr$ such that
\[
\gra(U_t)=\{(x,y)\in\rr^2:~d((x,y),\gra(M))<t\},
\]
where $d$ is the $\ell^1$-metric on $\rr^2$.
Clearly the following hold: $U_t$ has open graph,  $\bigcup_{t>0} U_t=\rr^2$, $\overline{\gra(U_s)}\subseteq \gra(U_t)$ for all $s<t$, and $\gra(M)=\bigcap_{t>0} U_t$. Moreover, for each $x\in\rr$, the set $U_t(x)$ is convex. Indeed, by Corollary \ref{closed-function}, there exists a family, $\{R_\lambda\}_{\lambda\in\Lambda}$, of connected subsets of $\rr$, such that $U_t(x)=\bigcup_{\lambda\in\Lambda}R_{\lambda}$ and $0\in\bigcap_{\lambda\in\Lambda}R_{\lambda}$. Hence, $U_t(x)$ is connected in $\rr$, that is, $U_t(x)$ is convex. Thus, by Theorem \ref{P-4}, there is a continuous function $\theta:\rr^2\to[0,1]$ such that
it is quasi-concave in its second argument and, for any $x\in\rr$, it holds:
\[
M(x)=\left\{ y\in\rr:~ \theta(x,y)=\max_{z\in\rr}\theta(x,z)\right\}.
\]
Since $M$ does not have compact values, we cannot apply Theorem 3.5 in \cite{LiLiFeng2022}  nor Theorem 1.3 in \cite{Ya08}.
%It is important to notice that
 %this is not a consequence of Theorem 2.1 in \cite{Ko97}, because $M$ does not have compact values.
\end{example}

 As we see below, it is easy to find other correspondences without compact values, admitting an inverse result.

\begin{proposition}\label{pro}
Let $X$ and $Y$ be two normed spaces, $D$ a non-empty and convex subset of $X$ and $M:D\tos Y$ be a correspondence with  non-empty convex and closed graph in $D\times Y$. Then, there exists a continuous function $\theta:D\times Y\to[0,1]$ such that the following two conditions hold:
\begin{itemize}
  \item [(i)] $\gra(M)=\{(x,y)\in D\times Y:~\theta(x,y)=\sup_{z\in Y}\theta(x,z)\}$, and
  \item [(ii)] $\theta(x,\cdot)$ is quasi-concave, for all $x\in D$.
\end{itemize}
\end{proposition}
The following lemma is Proposition 1.2.23 in \cite{Lucchetti2006}.
\begin{lemma}\label{lema}
 Let $H$ be a normed space, $C$ be a  closed, non-empty, and convex subset of $H$, and $f_C:H\to\rr$ be the function defined as $f_C(x)=d(x,C)$.  Then, $f_C$ is convex.
\end{lemma}
\begin{proof}[Proof of Proposition \ref{pro}]Indeed, let $C=\gra(M)$ and, for each $t>0$, define $U_t=\{(x,y)\in D\times Y:d((x,y),\gra(M))<t\}$. We have $U_t$ has open graph,  $\bigcup_{t>0} U_t=D\times Y$, $\overline{\gra(U_s)}\subset \gra(U_t)$ for all $s<t$, and $\gra(M)=\bigcap_{t>0} U_t$. In order to apply Theorem \ref{P-4}, it only remains to prove that, for each $x\in\rr$, the sets $U_t(x)=\{y\in Y:d((x,y),\gra(M))<t\}$ are convex. Let $y_1,y_2\in U_t(x)$ and $\lambda\in[0,1]$. By Lemma \ref{lema}, we have
$$
d(\lambda(x, y_1)+(1-\lambda)(x, y_2),\gra(M))\leq \lambda d((x,y_1),\gra(M))+(1-\lambda) d((x,y_2),\gra(M))<t,
$$
which completes the proof.
\end{proof}

Next, we introduce a simple correspondence with non-compact values, where, contrary to Komiya \cite{Ko97},   Li \emph{et al.} \cite{LiLiFeng2022}, and  Yamauchi \cite{Ya08} results, our Proposition \ref{pro}  applies.

\begin{example}  Let $M:(0,\infty)\to\rr$ be the correspondence defined by
$M(x)= [1/x,\infty)$. It is clear that $M$ has a non-empty convex and closed graph. Hence, by Proposition \ref{pro},
there exists a continuous function $\theta:X\times Y\to[0,1]$ such that the following two conditions hold:
\begin{itemize}
  \item [(i)] $\gra(M)=\{(x,y)\in X\times Y:~\theta(x,y)=\sup_{z\in Y}\theta(x,z)\}$, and
  \item [(ii)] $\theta(x,\cdot)$ is quasi-concave, for all $x\in X$.
\end{itemize}
\end{example}

Now, in a similar way to Theorem \ref{t5}, we present the following result. 

\begin{proposition}
Let $X$ be a non-empty paracompact space, $Y$ be  a non-empty convex and compact subset of a locally convex space, and
 $M:X\tos Y$ be a non-empty convex compact-valued and upper semicontinuous correspondence. Then,  the following two conditions are equivalents:
 \begin{itemize}
\item [(i)] there exists a continuous function $\theta:X\times Y\to[0,1]$, such that
\[
\gra(M)=\left\{(x,y)\in X\times Y:\theta(x,y)=\sup_{z\in Y}\theta(x,z)\right\},
\]
and the function $\theta(x,\cdot):Y\to[0,1]$  is quasi-concave, for each $x\in X$, and
\item [(ii)] $\gra(M)$ is a $G_\delta$ set.
\end{itemize}
\end{proposition}
\begin{proof}
Since $X\times Y$ is normal (c.f. Corollary 1.16, Chapter 3 in \cite{MN89}), 
by Theorem \ref{t5}, condition (i) implies condition (ii). 
Reciprocally, due to Theorem 1.3 in \cite{Ya08}, we have that condition (i) follows from condition (ii). 
%Suppose condition (i) holds. \RED{By Theorem \ref{t1}, $M$ is upper semicontinuous with compact and non-empty values and consequently by} 
%  the closed graph theorem, $\gra(\RED{M})$ is closed. \RED{Since} $X\times Y$ is normal (c.f. Corollary 1.16, Chapter 3 in \cite{MN89}), we have $\gra(\RED{M})$ is normal. This fact, along with Theorem 3.5 in \cite{LiLiFeng2022}, implies condition (a) in Theorem \ref{t5}. Consequently, $\gra(\RED{M})$ is a $G_\delta$ set in $X\times Y$ and the proof is complete. Reciprocally, condition (i) follows from condition (ii), by Theorem 1.3 in \cite{Ya08}. This completes the proof.
\end{proof}

Thanks to the previous result and Theorem 3.5 in \cite{LiLiFeng2022}, we have the following result.
\begin{proposition}
Let $X$ be a non-empty paracompact space, $Y$ be  a non-empty convex and compact subset of a locally convex space, and $M:X\tos Y$ be a non-empty convex compact-valued and upper semicontinuous correspondence. Then, the graph of $M$, $\gra(M)$, is a  $G_\delta$ set.
\end{proposition}

%===========

\section{Applications to generalized Nash games}\label{Applications}
%In this section, we present some remarks, about the minimax inequality by Ky Fan, and generalized Nash games. More precisely, we will use inverse maximum theorems to establish some equivalence between some famous results.
A \emph{Nash game}, \cite{Na51}, consists of $p$ players, each player $i$ controls the decision variable $x_i$, which belongs to a subset $C_i$ of a topological space $E_i$.
%$x_i\in C_i$ where $C_i$ is a subset of a locally convex topological vector space $E_i$.
 The ``total strategy vector'' is $x$,
 which will be often denoted by
 \[
  x=(x_1,\dots,x_i,\dots,x_p).
 \]
Sometimes we write $(x_i,x_{-i})$ instead of $x$ in order to emphasize the $i$-th player's variables within $x$, where $x_{-i}$ is the strategy vector of the other players.
 Player $i$ has a payoff function $\theta_i:C\to\rr$ that depends on all player's strategies, where $C=\prod_{i=1}^p C_i$.
 Given the strategies $x_{-i}\in C_{-i}=\prod_{j\neq i}C_j$ of the other players, the aim of player $i$ is to choose a strategy $x_i$ solving the problem $P_i(x_{-i})$:
\begin{align*}
\max_{ x_i }\theta_i(x_i,x_{-i}) ~\mbox{ subject to }~x_i\in C_i.
\end{align*}
 A vector $\hat{x}\in C$ is a \emph{Nash equilibrium} if, for all $i\in\{1,\dots,p\}$, $\hat{x}_i$ solves $P_i(\hat{x}_{-i})$. We denote by $NG(\theta_i,C_i)$ the set of Nash equilibria associated to the functions $\theta_i$ and the sets $C_i$.

In a generalized Nash game,  each player's strategy must belong to a set identified by the correspondence $K_i: C_{-i}\tos C_i$ in the sense that the strategy space of player $i$ is $K_i(x_{-i})$, which depends on the rival player's strategies $x_{-i}$.
 Given the strategy $x_{-i}$,  player $i$ chooses a strategy $x_i$ such that it solves the following problem $GP(x_{-i})$
 \begin{equation*}
\max_{x_i}\theta_i(x_i,x_{-i})~\mbox{ subject to }~x_i\in K_i(x_{-i}).%\tag{GNEP}
\end{equation*}
Thus, a \emph{generalized Nash equilibrium} is a vector $\hat{x}\in C$ such that
 the strategy $\hat{x}_i$ is a solution of the problem $GP(\hat{x}_{i})$, for any $i\in\{1,\dots,p\}$. We denote by $GNG(\theta_i,K_i)$ the set of generalized Nash equilibria associated to the functions $\theta_i$ and the correspondences $K_i$. Thus,
\[
GNG(\theta_i,K_i)=\{\hat{x}\in C:~\hat{x}\in NG(\theta_i, K_i(\hat{x}_{-i})\}.
\]

It is clear that any Nash game is a generalized Nash game. However, the last one is more complex due to the strategy set of each player depends of the strategy of his/her rivals.

\subsection{An inverse Nash theorem}

It is not difficult to see the following: %is given by the set
$$
%\hat{X}
GNG(\theta_i,K_i)=\bigcap_{i=1}^p\left\{x\in C:\theta_i(x)=\max_{z_i\in K_i(x_{-i})}\theta_i(z_i,x_{-i})\right\}.
$$
Indeed, we notice that, if for any player $i$, we consider its argmax correspondence $M_i:C_{-i}\tos C_i$; then $\left\{x\in C:\theta_i(x)=\max_{z_i\in K_i(x_{-i})}\theta_i(z_i,x_{-i})\right\}=\gra(M_i)$. Thus,
\[
GNG(\theta_i,K_i)=\bigcap_{i=1}^p\gra(M_i).
\]

The following result establishes that $GNG(\theta_i,K_i)$ is a $G_\delta$ set, under suitable assumptions.
\begin{proposition}
For each $i\in\{1,\dots,p\}$, let $K_i: C_{-i}\tos C_i$ be a continuous correspondence  with non-empty and compact values. If for each $i\in\{1,\dots,p\}$, the payoff function $\theta_i:C\to\rr$ is continuous, then $GNG(\theta_i,K_i)$ is a $G_\delta$ set.
\end{proposition}
\begin{proof}
For each $i\in\{1,\dots,p\}$, let $m_i:C_{-i}\to\mathbb{R}$ be a function defined as
\[
m_i(x_{-i})=\max_{x_i\in K_i(x^{-i})}\theta_i(x_i,x_{-i}).
\]
The Berge maximum theorem, Theorem \ref{t1}, implies that $m_i$ is continuous. Thus, the result follows from
\[
GNG(\theta_i,K_i)=\bigcap_{n=1}^\infty\bigcap_{i=1}^p\{x\in C:\theta_i(x)>(1-1/n)m_i(x_{-i})\}.
\]
\end{proof}

It is important to notice that in the previous result, each payoff function $\theta_i$ is continuous on $C$, but in order to apply Berge's maximum theorem we just need that $\theta_i$ be continuous on $\gra(K_i)$.

Now, we are interested in the inverse problem, which consists of finding payoff functions $\theta_1,\dots,\theta_p$ giving as solution  a predetermined strategy set, $\hat{X}$,  of the game. In other words, given the correspondences $K_i$ and a set $\hat{X}\subseteq C$ we want to find payoff functions $\theta_i$ such that $\hat{X}=GNG(\theta_i,K_i)$.

We present the following example to illustrate the previous problem.

\begin{example}\label{exa-inverse}
Consider $C_1=C_2=[0,1]$ and the correspondences $K_1,K_2:[0,1]\tos[0,1]$ defined as
\[
K_1(y)=[0,y]\mbox{ and }K_2(x)=[0,1-x].
\]
Figure \ref{F1} shows their graphs and consider $\hat{X}=\gra(K_1)\cap \gra(K_2)$.
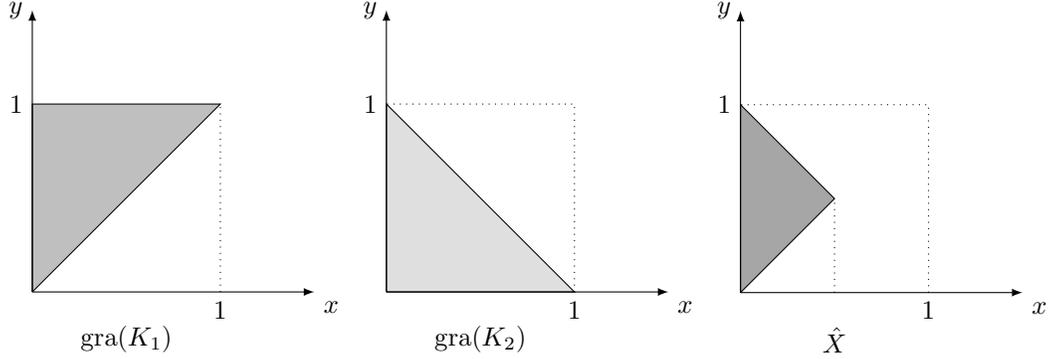
\begin{figure}[h!]
\centering
\begin{tikzpicture}[>=latex,scale=2.5]
\draw[fill=gray!50](0,0)--(1,1)--(0,1)--(0,0);
\draw[->](0,0)--(1.5,0)node[below right]{$x$};
\draw[->](0,0)--(0,1.5)node[left]{$y$};
\draw[dotted](1,0)--(1,1);
\draw(1,0)node[below]{$1$};
\draw(0,1)node[left]{$1$};
\draw(0.5,-0.25)node[]{$\gra(K_1)$};
\end{tikzpicture}
\begin{tikzpicture}[>=latex,scale=2.5]
\draw[fill=gray!25](0,0)--(0,1)--(1,0)--(0,0);
\draw[->](0,0)--(1.5,0)node[below right]{$x$};
\draw[->](0,0)--(0,1.5)node[left]{$y$};
\draw[dotted](1,0)--(1,1)--(0,1);
\draw(1,0)node[below]{$1$};
\draw(0,1)node[left]{$1$};
\draw(0.5,-0.25)node[]{$\gra(K_2)$};
\end{tikzpicture}
\begin{tikzpicture}[>=latex,scale=2.5]
\draw[fill=gray!70](0,0)--(0,1)--(0.5,0.5)--(0,0);
\draw[->](0,0)--(1.5,0)node[below right]{$x$};
\draw[->](0,0)--(0,1.5)node[left]{$y$};
\draw[dotted](1,0)--(1,1)--(0,1);
\draw[dotted](0.5,0)--(0.5,0.5);
\draw(1,0)node[below]{$1$};
\draw(0,1)node[left]{$1$};
\draw(0.5,-0.25)node[]{$\hat{X}$};
\end{tikzpicture}
\caption{Graphs of the sets $\gra(K_1)$, $\gra(K_2)$, and $\hat{X}$. }\label{F1}
\end{figure}

Now, we ask then do there are continuous functions $\theta_1,\theta_2:[0,1]\times[0,1]\to\rr$ such that $\hat{X}=GNG(\theta_i,K_i)$? In this case, the answer is positive. Indeed,
we can consider the functions $\theta_1,\theta_2:[0,1]\times[0,1]\to\rr$ defined as
\[
\theta_1(x,y)=1-d((x,y),\gra(K_1))\wedge 1\mbox{ and }\theta_2(x,y)=1-d((x,y),\gra(K_2))\wedge 1,
\]
where $d$ denotes the Euclidean metric in $\rr^2$.
Clearly $\theta_1$ and $\theta_1$ are continuous on $[0,1]\times[0,1]$. Moreover, it is not difficult to show that $\hat{X}=GNG(\theta_i,K_i)$, for $i\in\{1,2\}$.
\end{example}

The following result gives sufficient conditions to obtain a positive answer to this new problem.

\begin{proposition}\label{prop}
For each $i\in\{1,\dots,p\}$, let $K_i: C_{-i}\tos C_i$ be a continuous correspondence  with non-empty and compact values, and normal graph; and let $\hat{X}\subseteq \bigcap_{i=1}^p\gra(K_i)$. If $\hat{X}$ is a $G_\delta$ closed set then there exist
continuous functions $\theta_i:\gra(K_i)\to[0,1]$ such that
\[
\hat{X}=GNG(\theta_i,K_i).
\]
\end{proposition}
\begin{proof}
Suppose there exists a non-increasing sequence of open subsets of $C$, $\{U_n\}_{n\in\mathbb{N}}$, such that $\hat{X}=\bigcap_{n\in\mathbb{N}}U_n$. For each $i\in\{1,\dots,p\}$, since $\gra(K_i)$ is normal,  for each $n\in\mathbb{N}$, there exists a Urysohn function $\theta_{n,i}:\gra(K_i)\to[0,1]$ such that $\theta_{n,i}\equiv 1$ on $\hat{X}$ and $\theta_{n,i}\equiv 0$ on $\gra(K_i)\setminus U_n$. Let $\theta_{i}:\gra(K_i)\to[0,1]$ be defined as
\[
\theta_{i}(x)=\sum_{n=1}^\infty\frac{1}{2^n}\theta_{n,i}(x).
\]
We have, $\theta_{i}$ is continuous, $\theta_{i}^{-1}(\{1\})=\hat{X}$ and therefore
$\hat{X}=GNG(\theta_i,K_i)$, which completes the proof.
\end{proof}

\begin{remark}
When $C$ is a subset of a metric space with metric $d$ and the graphs $\gra(K_i)$ are closed, functions $\theta_i:\gra(K_i)\to[0,1]$ in Proposition \ref{prop} can be defined as $\theta_i(x)=1-d(x,\gra(K_i))\wedge 1$, for all $i\in\{1,\dots,p\}$.
\end{remark}

We give an interpretation of the previous problem as follows: agents could have a priori a set of strategies, which, due to certain restrictions, are not possible to be replaced. In this case, they only go to the market to participate in businesses whose payment functions favor the application of their strategies. An inverse Nash theorem provides the existence of appropriate actions giving place to predetermined results.

\subsection{Equivalent results}

This section aims to show the equivalence between two known results. First, we stated a classical result concerning the existence of Nash equilibria, inspired  by Debreu, Fan, and Glicksberg.

\begin{theorem}\label{D}
Suppose for each $i\in\{1,2,\dots,p\}$, $C_i$ is  a compact, convex and non-empty subset of a locally convex topological vector space $E_i$,
the payoff function $\theta_i$ is continuous and the correspondence $M_i:C_{-i}\tos C_i$, defined as
\[
M_i(x_{-i})=\left\{x_i\in C_i:~\theta_i(x_i,x_{-i})=\max_{z_i\in C_i}\theta_i(z_i,x_{-i})\right\},
\]
is convex-valued. Then,
the set $NG(\theta_i,C_i)$ is non-empty.

\end{theorem}

The following result is due to Arrow and Debreu \cite{AD54}.  We state it in the setting of locally convex spaces and  slightly modify the convexity condition as follows.

\begin{theorem}\label{A-D}
Suppose for each $i\in\{1,2,\dots,p\}$, $C_i$ is compact, convex and non-empty subset of a locally convex topological vector space $E_i$, and the following three conditions hold:
\begin{itemize}
\item[(i)] the payoff function $\theta_i$ is continuous,
\item[(ii)] the correspondence $K_i$ is continuous  with convex, closed and non-empty values, and
\item[(iii)] the correspondence $M_i:C_{-i}\tos C_i$ defined as
\[
M_i(x_{-i})=\left\{x_i\in C_i:~\theta_i(x_i,x_{-i})=\max_{z_i\in K_i(x_{-i})}\theta_i(z_i,x_{-i})\right\}
\]
is  convex-valued.
\end{itemize}
Then, the set $GNG(\theta_i,K_i)$ is non-empty. %there exists at least a generalized Nash equilibrium.
\end{theorem}

\begin{remark}\label{re1}
In the original version of Theorems \ref{D} and \ref{A-D}, the quasi-concavity of $\theta_i$ in $x_i$ for all players was assumed  by Debreu, Fan and Glickberg for Theorem \ref{D}, and Debreu and Arrow for Theorem \ref{A-D}. However this assumption implies the maps $M_i$ are convex-valued in both results.
Moreover, since Theorem \ref{A-D} was initially proved in finite dimensional spaces, this is a generalized version of the original result.
\end{remark}

In the following example we can see that a particular generalized Nash equilibrium problem can be considered as a classical Nash game.

\begin{example}
Consider a generalized Nash game with two players with correspondences $K_1,K_2:[0,1]\tos[0,1]$ defined as in Example \ref{exa-inverse} and payoff functions $\theta_1,\theta_2:[0,1]^2\to\rr$ defined as
\[
\theta_1(x,y)=y-x^2\mbox{ and }\theta_2(x,y)=2x-y^2.
\]
Thus, their maximizer correspondences are
\[
M_1(y)=\{0\}\mbox{ and }M_2(x)=\{0\},
\]
which are also upper semicontinuous with convex, compact, and non-empty values. Futhermore, we can consider the new functions $\vartheta_1,\vartheta_2:[0,1]^2\to\rr$ defined as
\[
\vartheta_1(x,y)=-x\mbox{ and }\vartheta_2(x,y)=-y,
\]
which are continuous and concave. Moreover, $ GNG(\theta_i,K_i)=NG(\vartheta_i,C_i)$.
\end{example}

In the line of Theorem 5.6 in \cite{Co21},
 we will show that it is possible to reformulate the generalized Nash game  as a classical Nash game.
\begin{theorem}\label{GNEP-NEP}
Suppose for each $i\in\{1,2,\dots,p\}$, $C_i$ is compact and non-empty, and the following two conditions hold:
\begin{itemize}
\item[(i)] the payoff function $\theta_i$ is continuous, and
\item[(ii)] the set-valued map $K_i$ is continuous  with  closed and non-empty values.
\end{itemize}
Then, there exist  continuous functions $\vartheta_i:C\to[0,1]$ such that
$GNG(\theta_i,K_i)=NG(\vartheta_i,C_i)$.
\end{theorem}
\begin{proof}
For each $i\in\{1,2,\dots,p\}$, by the Berge maximum theorem, Theorem \ref{t1}, there exists an upper semicontinuous correspondence with compact, convex, and non-empty values, $M_i:C_{-i}\tos C_i$, such that
\[
M_i(x_{-i})=\left\lbrace x_i\in K_i(x_{-i}):~\theta_i(x_i,x_{-i})=\max_{z_i\in K_i(x_{-i})}\theta_i(z_i,x_{-i})\right\rbrace.
\]
Now, by Proposition \ref{t7}, there exists a continuous function $\vartheta_i:C\to[0,1]$ such that
\[
M_i(x_{-i})=\left\lbrace x_i\in C_i:~\vartheta_i(x_i,x_{-i})=\max_{z_i\in C_i}\vartheta_i(z_i,x_{-i})\right\rbrace.
\]
Thus, $\hat{x}\in GNG(\theta_i,K_i)$ if, and only if, $\hat{x}\in NG(\vartheta_i,C_i)$.
\end{proof}

Another kind of reformulation is given by considering extended-real valued functions, that means associated to each player $i$, we define the function $\varphi_i:C\to\rr\cup\{-\infty\}$ by
\[
\varphi_i(x)=\theta_i(x)-\delta_{\gra(K_i)}(x),
\]
where $\delta_{\gra(K_i)}$ is the indicator function associated to the graph of $K_i$, that is $\delta_{\gra(K_i)}(x)=0$ if, $x\in \gra(K_i)$, and $\delta_{\gra(K_i)}=\infty$ if $x\notin \gra(K_i)$. We affirm that $ GNG(\theta_i,K_i)=NG(\varphi_i,C_i)$. Indeed, let $\hat{x}$ be an element of
$GNG(\theta_i,K_i)$, that is  for every $i$, $\hat{x}_i\in K_i(\hat{x}_{-i})$ and
\[
\theta_i(\hat{x})\geq \theta_i(x_i,\hat{x}_{-i}),\mbox{ for all }x_i\in K_i(\hat{x}_{-i}).
\]
Consequently, $\varphi_i(\hat{x})=\theta_i(\hat{x})$ and for any $x_i\notin K_i(\hat{x}_{-i})$ one has $\varphi_i(x_i,\hat{x}_{-i})=-\infty$. Thus, $\hat{x}\in NG(\varphi_i,C_i)$.
Reciprocally, let $\hat{x}\in NG(\varphi_i,C_i)$. Since $K_i$ is non-empty valued, we deduce that $\hat{x}_i\in K_i(\hat{x}_{-i})$. The result follows from the fact that  for any $x_i\in K_i(\hat{x}_{-i})$ we have $\varphi_i(x_i,\hat{x}_{-i})=\theta_i(x_i,\hat{x}_{-i})$.
However, this kind of reformulation does not guarantee the continuity of each function $\varphi_i$.

On the other hand, it is clear that Theorem \ref{A-D} implies Theorem \ref{D}. However,  they are actually equivalent. This is stated below.
\begin{theorem}\label{A-D-D}
Theorem \ref{D} implies Theorem \ref{A-D}.
\end{theorem}
\begin{proof}
This is  a consequence of Theorem \ref{GNEP-NEP} and
Theorem \ref{D}.
\end{proof}

\begin{remark}
Considering the original version of Theorems \ref{D} and \ref{A-D}, we can apply Theorem 3.5 in \cite{LiLiFeng2022} to show the previous result.
\end{remark}

\section{Application to fixed point theory}\label{FPT}
The authors in  \cite{YuEtAl16} showed that  Kakutani's fixed point theorem, Theorem \ref{KFFPT}, is consequence of Theorem \ref{D} on finite dimensional spaces. Moreover,
the Kakutani fixed point theorem admits an extension to locally convex spaces, by means of the Kakutani-Fan-Glicksberg theorem (see \cite{Fa52,Gl52}), which we state below.
\begin{theorem}\label{KFFPT}
 Let $C$ be a non-empty convex and compact subset of a Hausdorff locally convex topological vector space $Y$  and let $T:C\tos C$ be a correspondence. If $T$ is upper semicontinuous with convex, closed, and non-empty values, then the set $\{x\in C:~x\in T(x)\}$ is non-empty.
\end{theorem}
We will show that the implication given in  \cite{YuEtAl16}  is also true on locally convex spaces.

\begin{proposition}\label{D-KFFPT}
Theorem \ref{D} implies Theorem \ref{KFFPT}.
\end{proposition}
\begin{proof}
Assume that $T$ does not have a fixed point. Then the diagonal of $C\times C$, $D=\{(x,x)\in Y \times Y: x\in C\}$, is closed and disjoint to the graph of $T$. Thus, $\gra(T)\subset D^c$ and by Proposition 3.1 in \cite{Ya08}, there exists a continuous function $\theta:C\times C\to\rr$ such that $\theta(\gra(T))={1}$ and $\theta(D)=\{0\}$. Moreover, $\theta$ is quasi-concave in its second argument.
Let $P$ be a separating family of seminorms that generates the topology of $E$. For each $\rho\in P$, we consider the set
\[
F_\rho=\{(x,y)\in C\times C:~\theta(x,y)\geq \theta(x,z)\mbox{ and }\rho(x-y)\leq \rho(w-y),\mbox{ for all }w,z\in C\}.
\]
Clearly $F_\rho$ is compact. Furthermore, the family $\{F_\rho\}_{\rho}$ has the finite intersection property. Indeed, consider $\rho_1,~\dots,\rho_n\in P$ and consider the game with two players and their payoff functions $\theta_1,\theta_2:C\times C\to\rr$ defined by
\[
\theta_1(x,y)=-\sum_{i=1}^n\rho_i(x-y)\mbox{ and }\theta_2(x,y)=\theta(x,y).
\]
Since each function $\theta_j$ is continuous and quasi-concave in $x_j$, from Theorem \ref{D}, we deduce the existence of a Nash equilibrium $(\hat{x},\hat{y})$. Thi means
\[
\theta(\hat{x},\hat{y})\geq \theta(\hat{x},y)\mbox{ for all }y\in C,
\]
and
\[
\sum_{i=1}^n\rho_i(\hat{x}-\hat{y})\leq \sum_{i=1}^n\rho_i(x-\hat{y}),\mbox{ for all }x\in X.
\]
In this last inequality for $x=\hat{y}$ we obtain that $\rho_i(\hat{x}-\hat{y})=0$ for all $i\in\{1,2,\dots,n\}$. Consequently, $(\hat{x},\hat{y})\in \bigcap_{i=1}^n F_{\rho_i}$. Hence, there exists  $(\hat{x},\hat{y})\in\bigcap_{\rho\in P}F_\rho$ and this implies
$\rho(\hat{x}-\hat{y})=0$, for all $\rho\in P$. Thus $\hat{x}=\hat{y}$. Since $\hat{y}$ maximizes the function $\theta(\hat{x},\cdot)$, $\theta(\gra(T))=\{1\}$  and $\theta(D)=\{0\}$  we get a contradiction.
\end{proof}

Finally, inspired by Komiya \cite{Ko97}, we prove that the famous minimax inequality, due to Ky Fan \cite{Fa72}, implies Kakutani-Fan-Glicksberg's theorem

\begin{theorem}[Fan, 1972]\label{minmax}
Let $X$ be a compact  convex subset of a topological vector space. Let $f$ be a real-valued function defined on $X\times X$ such that
\begin{enumerate}
\item[(i)] for each $y\in X$, $f(\cdot,y)$ is lower semicontinuous; and
\item[(ii)] for each $x\in X$, $f(x,\cdot)$ is quasi-concave.
\end{enumerate}
Then, the minimax inequality
\[
\min_{x\in X}\max_{y\in X}f(x,y)\leq \max_{x\in X} f(x,x)
\]
holds.
\end{theorem}

\begin{proposition}\label{mmc->KFF}
Theorem \ref{KFFPT} is a consequence of Theorem \ref{minmax}.
\end{proposition}
\begin{proof} Let $T:C\tos C$ be a correspondence satisfying the assumptions of Theorem \ref{KFFPT}.
Thanks to Theorem 3.5 in \cite{LiLiFeng2022},  there exists a continuous function $\theta:C\times C\to[0,1]$ such that $\theta$ is quasi-concave in its second argument and
\[
T(x)=\left\lbrace y\in C:~ \theta(x,y)=\max_{z\in C}\theta(x,z)\right\rbrace, \mbox{ for all }x\in C.
\]
We consider the function $f:C\times C\to\rr$ defined as
\[
f(x,y)=\theta(x,y)-\theta(x,x).
\]
Clearly, $f$ vanishes on the diagonal of $C\times C$ and it satisfies all assumptions of Theorem \ref{minmax}. Thus, there exists $x_0\in C$ such that
$f(x_0,y)\leq 0$, for all $y\in C$. This means, $x_0$ maximizes $f(x_0,\cdot)$.
On the other hand, it is clear that $T(x)=\{y\in C:~ f(x,y)=\max_{z\in C}f(x,z)\}$.
Therefore, we deduce that $x_0\in T(x_0)$. This completes the proof.
\end{proof}

\section{Conclusions}\label{Conclusions}

Motivated by the works of Komita \cite{Ko97}, Park and Komiya \cite{PK01}, Aoyama \cite{Ao03}, Yamauchi \cite{Ya08}
 and Li \emph{et al.} \cite{LiLiFeng2022}, we present some inverse maximum theorems that are independent of those previously mentioned. Some introduced examples give an account of  this independence.
 As applications of our results, we first present an inverse maximum Nash theorem, second, we reformulate generalized Nash games as a classical Nash game under continuity assumption, and third, we prove the equivalence between famous equilibrium theorems and that they are equivalent to the well-known Kakutani-Fan-Glicksberg theorem and the famous minimax inequality due to Ky Fan.

%Motivated by the main result by Komiya in \cite{Ko97},  some inverse theorems associated with the classical maximum Berge Theorem are presented. These results extend to the infinity dimensional setting the mentioned theorem by Komiya. Moreover, an example is introduced, which shows that, although a hypothesis of Komiya result is not satisfied, one of our results applies. We state some equivalences between the existence of continuous objective functions and the property that the maximizers' correspondence.

%In the field of applications, we present an inverse maximum Nash theorem and prove the equivalence between famous equilibrium theorems, and that they are equivalent to the well-known Kakutani-Fan-Glicksberg theorem.

%\bibliographystyle{plain} \bibliography{References,Data-base}

\end{document}